\documentclass{article}
\usepackage{graphicx} 
\usepackage[top=3cm,bottom=3cm,left=3cm,right=3cm]{geometry}

\usepackage{palatino}
\usepackage{latexsym}

\usepackage[mathscr]{eucal}
\usepackage{amsmath}
\usepackage{amsthm}
\usepackage{amsfonts}
\usepackage{amssymb}
\usepackage{amscd}
\usepackage{color}
\usepackage{graphicx}
\usepackage{graphics}
\usepackage{pifont}
\usepackage{subfigure}
\usepackage[makeroom]{cancel}
\usepackage[normalem]{ulem}
\usepackage[dvipsnames]{xcolor}
\usepackage{tikz}
\usepackage{cite}
\theoremstyle{plain}
\newtheorem{theorem}{Theorem }

\newtheorem{lemma}[theorem]{Lemma}
\newtheorem{proposition}[theorem]{Proposition}

\newtheorem{corollary}[theorem]{Corollary}

\newcommand{\RR}{\mathcal{R}}
\newcommand{\R}{{\mathbb R}}
\newcommand{\N}{{\mathbb N}}

\newcommand{\C}{{\mathbb C}}

\newcommand{\la}{{\langle}}
\newcommand{\ra}{{\rangle}}
\newcommand{\grad}{{\nabla}}
\renewcommand{\div}{\operatorname{div}}

\title{Small-time controllability for the nonlinear Schrödinger equation on $\R^N$ via bilinear electromagnetic fields}
\author{Alessandro Duca,\footnote{Universit\'e de Lorraine, CNRS, INRIA, IECL, F-54000 Nancy, France (alessandro.duca@inria.fr)}\and{Eugenio Pozzoli\footnote{Univ Rennes, CNRS, IRMAR - UMR 6625, 35000 Rennes, France (eugenio.pozzoli@univ-rennes.fr)}}}

\begin{document}

\maketitle
\begin{abstract}
We address the small-time controllability problem for a nonlinear Schr\"odinger equation (NLS) on $\R^N$ in the presence of magnetic and electric external fields. We choose a particular framework where the equation becomes $i\partial_t \psi = [-\Delta+u_0(t)h_{\vec{0}}+\la u(t), P\ra +\kappa|\psi|^{2p}]\psi$. Here, the control operators are defined by the zeroth Hermite function $h_{\vec{0}}(x)$ and the momentum operator $P=i\grad$. In detail, we study when it is possible to control the dynamics of (NLS) as fast as desired via sufficiently large control signals $u_0$ and $u$. We first show the existence of a family of quantum states for which this property is verified. Secondly, by considering some specific states belonging to this family, as a physical consequence we show the capability of controlling arbitrary changes of energy in bounded regions of the quantum system, in time zero. Our results are proved by exploiting the idea that the nonlinear term in (NLS) is only a perturbation of the linear problem when the time is as small as desired. 
The core of the proof, then, is the controllability of the bilinear equation which is tackled by using specific non-commutativity properties of infinite-dimensional propagators.
\end{abstract}



\textbf{MSCcodes.} 35Q55, 81Q93, 93B05

\section{Introduction}
\subsection{The 
model}
The dynamics of a quantum particle moving in the Euclidean space subject to external electric and magnetic fields, and state nonlinearity, can be described via the following nonlinear Schr\"odinger equation 
\begin{equation}\label{eq:schro_intro}
i \frac{\partial}{\partial t}\psi=\left[-\big(\div+iA\big)\circ\big(\grad+iA\big)  +E+\kappa|\psi|^{2p}\right]\psi,
\end{equation}
where $\psi=\psi(x,t), (x,t)\in\mathbb{R}^N\times[0,T]$, $\kappa\in\mathbb{R}$, $p\in\N$ and $T>0$.  The multipolarized magnetic field is represented by the $\mathbb{R}^N$-valued function $A$ which, in our model, only depends on time $$A=(A_1,...,A_N):[0,T]\longrightarrow \R^N.$$ The scalar electric field, instead, is modeled by the $\mathbb{R}$-valued function $E$ depending on both time and space: $$E:[0, T]\times \R^N\longrightarrow \R.$$
The particle is represented by the quantum state $\psi$ evolving in the unit sphere of $L^2(\R^N,\C)$. An example of such evolution is the celebrated Gross-Pitaevskii equation when $p=1$.
This paper aims to investigate small-time controllability properties for \eqref{eq:schro_intro} via suitable electromagnetic fields. We consider suitable quantum states and we show that is possible to control them as fast as desired via sufficiently large electromagnetic fields $A$ and $E$. 

\smallskip

In this paper, we consider the equation \eqref{eq:schro_intro} in a particular case where it can be rewritten as the following nonlinear Schr\"odinger equation on the Hilbert space $H^s(\mathbb{R}^N,\mathbb{C})$,
\begin{equation}\label{eq:schro}\tag{NLS}\begin{cases}
i \dfrac{\partial}{\partial t}\psi(x,t)=\left[-\Delta+u_0(t)h_{\vec{0}}(x)+\la u(t), P\ra +\kappa|\psi(x,t)|^{2p}\right]\psi(x,t),\quad & \\
\psi(\cdot,0)=\psi_0(\cdot)\in H^s(\mathbb{R}^N,\mathbb{C}), \quad (x,t)\in \mathbb{R}^N\times[0,T],&
\end{cases}
\end{equation}
where $u_0:[0,T]\longrightarrow \R$ and $u=(u_1,...,u_N):[0,T]\longrightarrow \R^N$ and
$$\Delta=\sum_{l=1}^N\frac{\partial^2}{\partial {x_l}^2},\quad  \quad P=i \grad=i\left(\frac{\partial}{\partial {x_1}},\dots,\frac{\partial}{\partial {x_N}}\right), \quad \quad h_{\vec{0}}(x)=\pi^{-\frac{N}{4}}e^{-\frac{|x|^2}{2}}.$$
These are respectively the kinetic energy operator (that is, the Laplacian $\Delta$), the momentum operator $P$ (its components will also be denoted as $P_j=i\partial/\partial {x_j}$), and the zeroth $N$-dimensional Hermite function $h_{\vec{0}}$. 
The equation \eqref{eq:schro} can be rewritten in the form of \eqref{eq:schro_intro} by choosing (see Section \ref{sec:preliminaries} for further details)
$$A(t)=
-\frac{1}{2}(u_1(t),...,u_N(t)),\quad \quad \quad E(t,x)=u_0(t)h_{\vec{0}}(x)-\frac{1}{4}|u|^2.$$

\smallskip

In the framework of \eqref{eq:schro}, the particle is experiencing two external control fields: the first one is coupled to the momentum of the particle $P$, and the second one is coupled to a function of its position $x$, which is $h_{\vec{0}}(x)$. Notice that the choice of the control operator $h_{\vec{0}}$ is mathematically convenient (see Section \ref{sec:saturation} for more details) and is realistic from a physical point of view, as the dipolar interaction is concentrated on small values of the position $x$, and becomes weak for large $x$.
The time-dependent functions $u_0,...,u_N$ can be freely chosen among the piecewise constant functions, and play the role of control laws which steer the quantum dynamics toward desired targets.

\subsection{The main results}
 The first main result of the work is a specific small-time approximate controllability property which is stated in the following theorem.
\begin{theorem}\label{thm:main-result}
Let $s,p\in\N$, $s>N/2$, and $\kappa\in\R$. Consider any initial state $\psi_0\in H^s(\mathbb{R}^N,\mathbb{C})$ and any $\phi\in H^{2s}(\mathbb{R}^N,\mathbb{R})$. Then, for any positive error and time $\varepsilon,T>0$, there exist a smaller time $\tau\in[0,T)$ and piecewise constant controls $(u_0,u):[0,\tau]\to \mathbb{R}^{N+1}$ such that the solution $\psi(t,(u_0,u),\psi_0)$ of \eqref{eq:schro} with initial condition $\psi_0$ and controls $(u_0,u)$ satisfies
$$\|\psi(\tau,(u_0,u),\psi_0)-e^{i\phi}\psi_0\|_{H^s(\R^N)}<\varepsilon. $$
\end{theorem} 

Theorem \ref{thm:main-result} allows to steer, as fast as desired, any sufficiently smooth quantum state to any other state obtained by multiplying it with an imaginary exponential function. Notice that this operation is not just a change of phase since the function $\phi$ is not a constant. We consider the parameter $s>N/2$ so that the equation \eqref{eq:schro} is locally-in-time well-posed in the space $H^s(\R^N,\C)$, and a final phase $\phi\in H^{2s}(\R^N,\R)$ so that $e^{i\phi}\psi_0\in H^s(\R^N,\C)$.  

Notice that already in the linear case (i.e. when $\kappa=0$) Theorem \ref{thm:main-result} is new: in this case, the result is also valid in $ H^s(\R^N,\C)$ for $s\leq N/2$ since the system is well-posed also in this lower regular cases.

\smallskip

As a physical consequence, Theorem \ref{thm:main-result} implies the capability of controlling arbitrary changes of energy in bounded regions of the quantum system, in time approximately zero. The energy of a state here is defined in the linear case, but we state the result in the more general nonlinear setting. In the linear case $\kappa=0$, we define the energy of a state $\psi$ in a region $S\subset \R^N$ as $-\la\Delta\psi,\psi\ra_{L^2(S)}$. Given a bounded set $S\subset \R^N$ with finite Lebesgue measure $|S|<\infty$, we introduce the states
\begin{equation*}
\phi_{\xi,S}=\frac{\rho_S}{|S|}e^{i\xi x},\quad \xi\in\mathbb{R}^N,
\end{equation*} 
where $\rho_S$ is any smooth function with compact support such that $S\subset{\rm supp}(\rho_S)$ and $\rho_S(x)=1$ on $S$. Notice that, in the linear case, $\phi_{\xi,S}$ has energy $|\xi|^2$ in $S$. We then have the following result.
\begin{corollary}\label{thm:eigenmodes}
Let $s,p\in\N$, $s>N/2$, and $\kappa\in\R$. Let $S\subset \R^N$ with $|S|<\infty$, and $\xi,\nu\in\mathbb{R}^N$ be two frequencies. Then, for any positive error and time $\varepsilon,T>0$, there exist a smaller time $\tau\in[0,T)$ and a piecewise constant control $(u_0,u):[0,\tau]\to \mathbb{R}^{N+1}$ such that the solution $\psi(t,(u_0,u),\phi_{\xi, S})$ of \eqref{eq:schro} with initial condition $\phi_{\xi, S}$ and control $(u_0,u)$ satisfies
$$\|\psi(\tau,(u_0,u),\phi_{\xi, S})-\phi_{\nu, S}\|_{H^s(S)}<\varepsilon. $$
\end{corollary}


\smallskip

Similar results can be found in the work \cite{duca-nersesyan} by Nersesyan and the first author, and in \cite{small-time-molecule,small-time-wave} by Chambrion and the second author. Nevertheless, they all deal with different frameworks from the one considered here. 
From this perspective, the main novelties of Theorem \ref{thm:main-result} are the following. 

\begin{itemize}

    \item Bilinear controls, not only for the Schr\"odinger equation, but also for the heat or the wave equations, are usually studied on compact manifolds. Indeed, most of the classical techniques rely upon spectral techniques needing a drift with point spectrum only. 
    Our approach, inspired by \cite{duca-nersesyan,small-time-molecule,small-time-wave}, does not require such features and can also be applied in systems where the drift has a continuous spectrum as in (NLS). 

\item Up to the recent works \cite{duca-nersesyan,small-time-molecule,small-time-wave}, classical approximate controllability results have been proved in large times. 
In our main results, the controllability is as fast as desired and the obvious price to pay is that the corresponding control amplitudes are more and more large. 

\item The controllability properties of Theorem \ref{thm:main-result} and Corollary \ref{thm:eigenmodes} are ensured despite the non-linear behaviour of the dynamics. Up to \cite{duca-nersesyan}, most of the existing works on the subject can only deal with the linear case. Here, we improve the techniques from \cite{small-time-molecule,small-time-wave} to treat also non-linear equations.

\end{itemize}

\subsection{The technique}
The control strategy to show Theorem \ref{thm:main-result} is explicit, and consists in applying large controls in short time intervals (which is natural, since we want to control the system in small times). These kinds of techniques are well-known in finite-dimensional geometric control, where large controls on short time intervals are usually considered to avoid the effect of the drift on the dynamics \cite{jurdje-kupka,glaser-brockett,domenico}. 

\smallskip

From a PDE perspective, this technique is inspired by the work of the first author and Nersesyan \cite{duca-nersesyan}, previously introduced, which treats the bilinear control of nonlinear Schr\"odinger equations on tori using low mode forcing. There, the authors leveraged the drift term to prove small-time approximate controllability among eigenstates. In this paper, besides considering a different unbounded model, we exploit different properties of non-commutativity of infinite-dimensional propagators: instead of leveraging the drift generated by the Laplacian $-\Delta$, we get rid of it and exploit the momentum operator $P$ to control the quantum system.

\smallskip

More precisely, if one forgets about the nonlinearity, the main idea of the proof is to consider the following small-time limit of conjugated dynamics:
\begin{equation}\label{eq:limit-intro}
\lim_{\tau\to 0} e^{ih_{\vec{0}}/\tau}e^{-i\tau P_j}e^{-ih_{\vec{0}}/\tau}\psi_0=e^{-P_jh_{\vec{0}}}\psi_0.
\end{equation}
This allows to gain a new non-directly accessible direction where the system can be steered: in analogy to finite-dimensional systems, we may regard $h_{\vec{0}}$ and $P_j$ as two linearly independent and directly accessible directions for the control system; then 
$$-P_jh_{\vec{0}}=i\pi^{-N/4}x_je^{-|x|^2/2}$$
is the first Hermite function in the $x_j$-variable and the zeroth in the other variables. Hence we generate a linearly independent direction which was not directly accessible. This new direction is defined by (minus) the commutator of the first two (an operation also referred to as Lie bracket in geometric control): 
$$[P_j,h_{\vec{0}}]\psi_0:=(P_jh_{\vec{0}}-h_{\vec{0}}P_j)\psi_0=-i\pi^{-N/4}x_je^{-|x|^2/2}\psi_0.$$
 We then iterate this strategy to generate any linear combination of Hermite functions (this procedure is usually called \emph{saturation} and has been introduced by Agrachev and Sarychev for controlling Navier-Stokes equations \cite{navier-stokes}, and we refer also to their recent work \cite{agrachev-ensemble} where Hermite functions are specifically considered for saturating purposes): being this set dense, we obtain the small-time approximate controllability property stated in Theorem \ref{thm:main-result} in the linear case. The result in the nonlinear case then follows by showing that the nonlinearity is a perturbation that does not influence the small-time limit \eqref{eq:limit-intro} (we refer to Proposition \ref{lemma:limit-nonlinear1} for a precise statement in the nonlinear case). 
\subsection{
Other literature}

The study of the controllability properties of bilinear Schr\"odinger equations plays an important role in many applications, due to the relevance of quantum phenomena in physics, chemistry, engineering, and information science. In particular, much attention has been devoted to this kind of system due to its fundamental relevance in the development of quantum technologies \cite{Glaser2015}.

\smallskip

From a mathematical point of view, 
many controllability results 
have been obtained in the last two decades. Local exact controllability of Schrödinger equations in higher Sobolev norms was firstly proven by Beauchard in \cite{beauchard1}. For local exact controllability results in the presence of state nonlinearity, see also the papers \cite{laurent,beauchard-2015}. We also mention the recent paper of Bournissou \cite{bournissou} on the local exact controllability of bilinear Schrödinger equations with Lie bracket techniques. For the global approximate controllability, the first achievements were made by Adami, Boscain, Chambrion, Mason, and Sigalotti \cite{adami,BCMS} and by Nersesyan \cite{nersesyan}. All of those works rely on the discreteness of the drift spectrum. 

\smallskip

Different quantum controllability properties (e.g., control between bound states) in the presence of a continuous spectrum in the drift as in the case of our work were studied by Mirrahimi \cite{mirra2}, and Chambrion \cite{chambrion}. 

\smallskip

The problem of minimizing the controllability time, and more in general proving small-time controllability properties, 
is of prime importance in quantum mechanics, with very few results available in the PDE setting. Boussaïd, Caponigro, and Chambrion \cite{minimal-time-thomas} showed that a particular bilinear conservative equation on the circle is globally approximately controllable in small times. Beauchard, Coron, and Teismann \cite{minimal-time-coron,minimal-time-approximate} proved an obstruction to small-time approximate controllability for bilinear Schr\"odinger equations with sub-quadratic uncontrolled potential (which is instead approximately controllable in large times): the main difference between their obstruction and the positive result obtained in Theorem \ref{thm:main-result} is in the choice of the control operators. More precisely, the crucial saturation obtained here in Section \ref{sec:saturation} is based on the properties of the function $h_{\vec{0}}$, and cannot hold if one considered instead only functions such as $x_j$ as control operators. We also refer to the recent work \cite{coron-small-semiclassical} by Coron, Xiang, and Zhang, where the geometric technique of saturation through low mode forcing is used for small-time control of semiclassical
bilinear Schr\"odinger equations.

We additionally notice that the capability of controlling quantum evolutions by means of their momentum (instead of the more standard position operator) was speculated e.g. in the work of Boscain, Mason, Panati, and Sigalotti \cite[Section I.C]{panati} on spin-boson systems, and the present paper thus provides some insights also in this direction.

We conclude with a comment on the results obtained in this paper. The problem of global small-time approximate controllability of (NLS) (between more general states than the ones obtained in Theorem \ref{thm:main-result}) remains open, and we plan to elucidate it in future investigations. We do not expect the family of small-time approximately reachable states described in Theorem \ref{thm:main-result} to be optimal: the techniques and the results presented here can probably be used to show additional small-time controllability properties of Schrödinger equations, and thus enlarge the above-mentioned family. Nevertheless, with the current techniques, we are still not able to conjecture any positive (nor negative) results concerning the (more general) global small-time approximate controllability problem of (NLS).




\subsection{Structure of the paper} The paper is organized as follows. In Section \ref{sec:preliminaries}, we recall the notion of solution for \eqref{eq:schro}. In Section \ref{sec:small-time}, we prove the small-time limit of conjugated dynamics \eqref{eq:limit-intro}, and related ones, in the linear case. In Section \ref{sec:NLS} we extend the validity of the small-time limits to the nonlinear case. In Section \ref{sec:saturation}, we show a density property of the directions where the control system can be steered. We conclude in Section \ref{sec:proof} by proving Theorem \ref{thm:main-result}.

\textbf{Acknowledgments.} The authors would like to thank Ugo Boscain, Nabile Boussa\"id, Thomas Chambion, David Dos Santos Ferreira, and Vahagn Nersesyan for fruitful conversations.

E.P. acknowledges support by the PNRR MUR project PE0000023-NQSTI. This work was also part of the project CONSTAT, supported by the Conseil Régional de Bourgogne Franche- Comté and the European Union through the PO FEDER Bourgogne 2014/2020 programs, by the French ANR through the grant QUACO (ANR-17-CE40-0007-01) and by EIPHI Graduate School (ANR-17-EURE-0002).

\section{Preliminaries and well-posedness}\label{sec:preliminaries}

We start by showing the computations allowing to rewrite the equation \eqref{eq:schro_intro} in the form of \eqref{eq:schro} when 
\begin{equation}\label{fields}A(t)=
-\frac{1}{2}(u_1(t),...,u_N(t)),\quad \quad \quad E(t,x)=u_0(t)h_{\vec{0}}(x)-\frac{1}{4}|u|^2,
\end{equation}
as presented in the introduction. We firstly rewrite the equation \eqref{eq:schro_intro} in the form
\begin{equation}\label{eq:preliminaries_1}\begin{split}
i \frac{\partial}{\partial t}\psi&= -\div(\grad \psi)  - i\div(A(t)\psi)     - i\la A(t), \grad\psi\ra  +|A(t)|^2\psi+E(x,t)\psi+\kappa|\psi|^{2p}\psi.\\
\end{split}\end{equation}
Notice that the magnetic field $A(t)$ only depends on time and, then, we have
$$\div(A(t)\psi)= \div(A(t))\psi + \la A(t), \grad \psi\ra =\la A(t), \grad \psi\ra.$$
The last identity allows to rewrite the equation \eqref{eq:preliminaries_1} as follows 
\begin{equation*}\begin{split}
i \frac{\partial}{\partial t}\psi&= -\Delta\psi  - 2i\la A(t), \grad \psi\ra  +|A(t)|^2\psi+E(x,t)\psi+\kappa|\psi|^{2p}\psi\\
&= -\Delta\psi +  \la - 2A(t), i\grad \psi\ra  +\big(|A(t)|^2+E(x,t)\big)\psi+\kappa|\psi|^{2p}\psi.\\
\end{split}\end{equation*}
Finally, we can write $i\grad = P$, $-2A(t)=u(t)$ and $|A(t)|^2+E(x,t)=u_0(t)h_{\vec{0}}(x)$
thanks to \eqref{fields}, which lead to the nonlinear Schr\"odinger equation \eqref{eq:schro}.

\smallskip

In the next proposition, we recall a local-in-time well-posedness result of the nonlinear Schrödinger equation adapted to our setting. 


\begin{proposition}\label{prop:well}
Let $s,p\in\N$, $s>N/2$, and $\kappa\in\R$. For any $\psi_0\in H^s(\mathbb{R}^N,\mathbb{C})$, and any $(u_0,u)\in L^1_{\rm loc}((0,\infty),\R^{N+1})$, there exist a maximal time $T>0$ and a unique solution 
$$\psi(t)=e^{it\Delta-i\la \int_{0}^tu(s)ds,P\ra)}\psi_0-i\int_{0}^t\!\!\!e^{i(t-\tau)\Delta-i\la \int_0^{t-\tau}u(s)ds,P\ra}(u_0(\tau)\vec{h_0}+\kappa|\psi(\tau)|^{2p})\psi(\tau)d\tau,$$
in $C^0([0,T),H^s(\R^N))$, of the problem \eqref{eq:schro} with initial state $\psi_0$. 

Denote by $\RR$ the propagator of the Schr\"odinger equation \eqref{eq:schro} relating any initial state $\psi_0$ and any control $(u_0,u)$ to the corresponding solution $R(t,\psi_0,(u_0,u))=\psi(t)$. For any other $\psi_1\in H^s(\R^N,\C)$, let $T>0$ be the maximal time such that both $R(t,\psi_0,(u_0,u))$ and $R(t,\psi_1,(u_0,u))$ exist for times strictly smaller than $T$. Then, we have continuity w.r.t. the initial data, that is: there exists $C=C(u_0,t,N,s,p,\kappa,\vec{h_0})\geq 0$ such that  
 \begin{align*}
\|\RR(\cdot,\psi_0,(u_0,u)) -\RR(\cdot,\psi_1,(u_0,u))\|_{C^0([0,T),H^s(\R^N))}
\leq C  \|\psi_0-\psi_1\|_{H^s(\R^N)}
\end{align*}

\end{proposition}

The proof of Proposition \ref{prop:well} is classical, so we omit it: the key points are the unitarity of the operator $e^{ir\Delta+i\la \int_0^{r}u(s)ds,P\ra}$ in $H^s(\R^N)$, and the fact that $H^s(\R^N)$ is an algebra when $s>N/2$. Hence, one can derive Proposition \ref{prop:well} by applying abstract arguments such as \cite[Proposition 1.38]{TT-2006} (as it is done e.g. in \cite[Proposition 3.8]{TT-2006}).


\section{Small-time limits in the linear case}\label{sec:small-time}

This section aims to prove some small-time limits such as \eqref{eq:limit-intro}, presented in the introduction.

\begin{proposition}\label{lemma:main-tool}
Let $\psi_0\in H^s(\mathbb{R}^{N},\mathbb{C})$ and $\varphi\in H^{2s+1}(\mathbb{R}^{N},\mathbb{R})$. Then, the following limit holds w.r.t. the $H^s$-norm
$$\lim_{\tau\to 0^+}\exp\left( i \frac{\varphi}{\tau}\right)\exp(-i\tau P_j)\exp\left(- i \frac{\varphi}{\tau}\right)\psi_0=\exp(-  P_j\varphi)\psi_0,\quad j=1,\dots,N. $$

\end{proposition}
\begin{proof}
For any $\tau>0$, we   denote the solution  $\phi(t)=  	e^{i \tau^{-1}\varphi} e^{-i t P}  e^{- i \tau^{-1}\varphi}\psi_0$. We  introduce the following functions
 \begin{align}\label{3.2}
 	w(t)  =   e^{- t (P\varphi) }\psi_0^\tau,\  \ \ \ \ \ \ \   v(t)  = \phi(\tau t)-w(t)
 \end{align}
 where 
 $\psi_0^\tau \in H^{s+2}(\R^N)$ is a regularized initial state satisfying the properties 
\begin{align}\label{mollif}\|	\psi_0-\psi_0^\tau\|_{H^s}\to 0  \quad\,\,\, \text{ as $\tau\to 0^+$},\end{align}
and
 \begin{align}\label{3.3}
 & \|\psi_0^\tau\|_{H^s}  \leq C  \quad\quad\quad\|\psi_0^\tau\|_{H^{s+2}} \leq C   \tau^{-1/4}  \,\,\,\,\quad \text{for } \tau\leq 1,
 \end{align}
for some constant $C>0$ independent of $\tau>0$. We want show that $\phi(\tau)\xrightarrow{\tau\rightarrow 0^+}  e^{ -(P\varphi)}\psi_0 $ w.r.t the $H^s$ norm. It is equivalent to proving the limit
\begin{align}\label{lim}\|v(t=1)\|_{H^s}\xrightarrow{\tau\rightarrow 0^+}  0  .\end{align}
We observe that the function $v$ is the solution of the following equation
\begin{equation}\label{3.7}
\partial_t v
 =- i \tau P (v(t)+w(t)) - (P\varphi) v(t) 
 \end{equation}
with initial condition 
     \begin{equation}\label{3.8}
      v(0)=\psi_0-\psi_0^\tau.
      \end{equation}
      We consider $v(t)$ solution of \eqref{3.7} in $t\in [0,2]$. Denote by $Re(z)$ (resp., $Im(z)$) the real (resp., imaginary) part of a complex number $z$. Since $Im \la P v,v \ra_{H^s}=0$, and thanks to \eqref{3.3}, there exist $C_1,C_2>0$ independent of $\tau>0$ such that
\begin{align*}
\partial_t\| v\|_{H^s}^2&= 2 Re (i\ \tau \la P v,v \ra_{H^s})+ 2 Re (i\ \tau \la P w,v \ra_{H^s})+2 Re \la (P\varphi) v, v  \ra_{H^s}\\
&= -2 \tau Im\la P w,v \ra_{H^s}+ 2 Re \la (P\varphi) v, v  \ra_{H^s}\\
&\leq 2\tau \| w\|_{{H^{s+1}}}\| v\|_{H^s}  + \|\varphi\|_{H^{s+1}}\| v\|_{H^s}^2  
 \\
&\leq C_1(\tau^\frac{3}{4} + \| v\|_{H^s}  )\| v\|_{H^s}\\ 
&\leq C_2( \tau^\frac{1}{2}+(\tau +1)\| v\|_{H^s}^2 ).
\end{align*}
Notice that in the very last inequality, we used the Young inequality.
Now, the Gronwall inequality implies that
\begin{equation}\label{3.10}
\|v(t)\|_{H^s}^2 \le e^{ C_2(\tau +1)t} \big(C_2\tau^\frac{1}{2} t +\|	\psi_0-\psi_0^\tau\|_{H^s}\big)
\end{equation}
Finally, the limit \eqref{lim} is proved thanks to the inequality \eqref{3.10} since$$
\|v(1)\|_{H^s}^2 \leq e^{ C_2(\tau +1)} \big(C_2\tau^\frac{1}{2} +\|	\psi_0-\psi_0^\tau\|_{H^s}\big)
\to 0 \quad \text{as $\tau\to 0^+$}.
$$

\end{proof}


We shall also need some additional small-time limits.

\begin{lemma}\label{coro:main-tool-bis}
Let $\psi_0\in H^s(\mathbb{R}^{N},\mathbb{C})$. Let $(\delta_n)_{n\in\N}$ be a sequence of positive numbers such that $\delta_n\to 0$ as $n\to\infty$, and $(u_n)_{n\in\N}\in \ell^\infty(\N,\R)$. The following limit
holds \begin{align*}\lim_{n \to \infty}\left\|\exp\left(-i { \delta_n} \left(-\Delta+\frac{u_n}{{  \delta_n}} Q_j\right)\right)\psi_0-\exp(-i u_n Q_j)\psi_0 \right\|_{H^s(\mathbb{R}^N)}=0,\end{align*}
where $Q_0=h_{\vec{0}}$ and $Q_j=P_j$ for $j\geq 1$.
\end{lemma}
\begin{proof}
\textbf{Case $j>0$.} Using the Plancherel Theorem we have
\begin{align*}
&\left\|\exp\left(-i { \delta_n} \left(-\Delta+\frac{u_n}{{  \delta_n}} P_j\right)\right)\psi_0-\exp(-i u_n P_j)\psi_0 \right\|^2_{H^s(\mathbb{R}^N)}\\
=&\int_{\R^N}(1+|\xi|^2)^s \left|(e^{-i\delta_n|\xi|^2}-1)e^{-iu_n\xi_j}\widehat{\psi_0}(\xi)\right|d\xi,
\end{align*}
where $\widehat{\psi_0}$ denotes the Fourier transform of $\psi_0$. The statement then follows by dominated convergence.

\textbf{Case $j=0$.} For $\psi_0\in H^{s+2}(\R^N)$, we introduce
\begin{align*}
\phi(t)&=\exp\left(-i t \left(-\Delta+\frac{u_n}{{  \delta_n}} h_{\vec{0}}\right)\right)\psi_0,\\
w(t)&=\exp(-it u_n h_{\vec{0}})\psi_0,\\
 v(t)&=\phi(\delta_nt)-w(t). 
\end{align*}
We are left to prove that 
\begin{equation}\label{eq:v-limit}
\|v(t=1)\|_{H^s}\to 0,\quad n\to \infty.
\end{equation}
 We have
$$\partial_tv= i\delta_n\Delta(v+w)-iu_nh_{\vec{0}}v,\quad v(0)=0,$$
hence, since $Im\la \Delta v,v\ra_{H^s}=0$, there exists $C$ uniformly bounded w.r.t. $n$ such that
$$\partial_t\|v\|_{H^s}^2\leq C(\delta_n\|v\|_{H^s}+|u_n|\|v\|_{H^s}^2)\leq C\left(\frac{\delta_n^2}{2}+\left(\frac{\delta_n^2}{2}+|u_n|\right)\|v\|_{H^s}^2\right), $$
where in the last inequality we used Young inequality. Thanks to the Gronwall inequality, we get
$$\|v(t)\|_{H^s}^2\leq Ct\frac{\delta_n^2}{2}e^{Ct\left(\frac{\delta_n^2}{2}+|u_n|\right)},  $$
which implies \eqref{eq:v-limit}. For $\psi_0\in H^{s}(\R^N)$, the statement follows by density.
\end{proof}

The following is a direct consequence of Lemma \ref{coro:main-tool-bis}.
\begin{corollary}\label{nuovo1}
Let $\psi_0\in H^s(\mathbb{R}^N)$. Then, for any $u\in\mathbb{R}$, 
\begin{align}\label{nonlin}\lim_{\tau\to 0^+}\sup_{t \in (0,1)}\left \|\left(\exp \left(- i \tau t \left(-\Delta+\frac{u}{\tau} h_{\vec{0}}\right) \right)-\exp(-i t u h_{\vec{0}})\right)\psi_0\right \|_{H^s(\R^N)}^2 =0.\end{align}
\end{corollary}

\section{Small-time limits in the nonlinear case}\label{sec:NLS}
We are now ready to prove that the nonlinearity does not affect the small-time limits previously considered.
\begin{proposition}\label{lemma:limit-nonlinear1}
Let $s,p\in\N$, $s>N/2$, and $\kappa\in\R$. Consider any initial state $\psi_0\in H^s(\mathbb{R}^N)$. Then, for any $u\in\mathbb{R}$, 
the following limit holds in $H^s$
\begin{align*}\lim_{\delta\to 0^+}\mathcal{R}\Big(\delta,\psi_0,e_j \frac{u}{\delta}\Big)=\exp(-i uQ_j)\psi_0,\end{align*}
where $\{e_j\}_{j=0,.., N}$ is the standard basis of $\mathbb{R}^{N+1}$, $Q_0=h_{\vec{0}}$ and $Q_j=P_j$ for $j\geq 1$.
\end{proposition}
\begin{proof}
{\bf Case $j>0$}. Let $\delta>0$. Proposition \ref{prop:well} gives the solution $\psi(t)=\mathcal{R}\big(t,\psi_0,e_j \frac{u}{\delta}\big)$, well-defined on a maximal time interval $[0,T^\delta)$ with $T^\delta>0$. We denote $$\varphi(t)=\exp(-i t uQ_j )\psi_0 \ \ \ \ \ \text{and}\ \ \  \ \ \phi(t)=\psi(\delta t)-\varphi(t),$$ which is well-defined for $t<\delta^{-1} T^\delta $. We want to prove the existence of $\delta_0>0$ such that, for every $0<\delta <\delta_0 $, we have $1\in[0,\delta^{-1} T^\delta)$ and finally ensure the limit 
 $$\|\phi(t=1)\|_{H^s(\mathbb{R}^N)}\xrightarrow{\delta\rightarrow 0^+} 0.$$ 
 Using the Duhamel formula, we can write 
\begin{equation}\begin{split}\label{duhamel}\phi(t) =&\left(\exp\left(-i \delta t \left(-\Delta+\frac{u}{\delta} Q_j\right)\right)-\exp(-i t uQ_j)\right)\psi_0 \\
&
-i\kappa \int_0^{\delta t} e^{i(\delta t-\tau)\left(\Delta-\frac{u}{\delta} Q_j\right)}|\psi(\tau)|^{2p}\psi(\tau) d\tau .\end{split}
\end{equation}
We now study the $H^s$-norm of the integral term in \eqref{duhamel}. As a general remark, thanks to the Cauchy-Schwartz inequality, we have that for every $f\in L^2((0,t), L^2(\R^N,\C))$
\begin{equation}\label{eq:cauchy}
    \left\|\int_0^{ t} \!\!\!\!f(r) dr \right\|^2_{L^2(\mathbb{R}^N)} \leq {  t} \int_0^{  t} \|  f(r)  \|^2_{L^2(\mathbb{R}^N)}dr.
\end{equation}
Moreover, notice that the operator $-\Delta+\frac{u}{\delta} Q_j$ commutes with $-\Delta$, implying the unitarity of the operator $e^{it \left( \Delta-\frac{u}{\delta} Q_j\right)}$ on $H^s$.
Combining these properties with the fact that $H^s(\R^N)$ is an algebra for $s>N/2$ yields the existence of $C=C(N,s)>0$ such that
\begin{equation}\begin{split}\label{limit1}
    \left\|(-\Delta)^\frac{s}{2}\!\!\! \int_0^{\delta t} \!\!\!\!e^{i({\delta t}-\tau)\left( \Delta-\frac{u}{\delta} Q_j\right)}|\psi(\tau)|^{2p}\psi(\tau) d\tau \right\|^2_{L^2(\mathbb{R}^N)} 
   & \!\!\!\!\!\!\!\!\!\!\!\!\!\leq {\delta t} \int_0^{\delta t} \|(-\Delta)^\frac{s}{2} |\psi(\tau)|^{2p}\psi(\tau)\,\|^2_{L^2(\mathbb{R}^N)}d\tau\\
    &\!\!\!\!\!\!\!\!\!\!\!\!\!\leq C{\delta t} \int_0^{\delta t} \|\psi(\tau)\|^{4p+2}_{H^s(\mathbb{R}^N)}d\tau.\\
        \end{split}
\end{equation}
Hence, when $t<\delta^{-1} T^\delta $, we get (for some $C=C(N,s)>0$)
\begin{align}\label{eq:align}
\|\phi(t)\|_{H^s(\mathbb{R}^N)}^2&\leq 2\left\|\exp\left(-i \delta t \left(-\Delta+\frac{u}{\delta} Q_j\right)\right)\psi_0-\exp(-i t uQ_j)\psi_0 \right\|_{H^s(\mathbb{R}^N)}^2\nonumber\\
&+2C\kappa^2{\delta t}\int_0^{\delta t} \|\psi(\tau)\|_{H^s(\mathbb{R}^N)}^{4p+2}d\tau\nonumber\\
&= 2\left\|\exp\left(-i \delta t \left(-\Delta+\frac{u}{\delta} Q_j\right)\right)\psi_0-\exp(-i t uQ_j)\psi_0 \right\|_{H^s(\mathbb{R}^N)}^2\nonumber\\
&+2C\kappa^2{\delta^2 t}\int_0^{  t} \|\psi(\delta r)\|_{H^s(\mathbb{R}^N)}^{4p+2}dr\nonumber\\
&\leq 2\left\|\exp\left(-i \delta t \left(-\Delta+\frac{u}{\delta} Q_j\right)\right)\psi_0-\exp(-i t uQ_j)\psi_0 \right\|_{H^s(\mathbb{R}^N)}^2\nonumber\\
&+  2C\kappa^2{\delta^2t}\int_0^{  t} \|\varphi(r)\|_{H^s(\mathbb{R}^N)}^{4p+2}dr  +C{\delta^2t}\int_0^{  t} \|\phi(r)\|_{H^s(\mathbb{R}^N)}^{4p+2}dr.
 \end{align}

Denote $\varepsilon^\delta:=\sup\{t<\delta^{-1} T^\delta:\ \|\phi(t)\|_{H^s(\R^N)}<1 \}$. We prove the existence of $\delta_0>0$ such that, for every $0<\delta <\delta_0 $, we have $\delta^{-1} T^\delta>1$. We ensure this property by showing that, for $\delta_0>0$ small enough, we have $$\varepsilon^\delta>1,\ \ \ \ \ \ \ \text{for every}\ \ \ \ 0<\delta <\delta_0.$$ 
We proceed by contradiction: assume that, for every $\delta_0>0$, there exists $0<\delta <\delta_0 $ such that $\varepsilon^\delta\leq 1$. Thus, there exists at least a sequence of positive numbers $\delta_n\xrightarrow{n\rightarrow +\infty} 0 $ such that $\varepsilon^{\delta_n}\leq 1$ for every $n\in\N$. Notice that there exists $C'=C'(u,Q) >0$, such that $$\|\varphi(t)\|_{H^s(\mathbb{R}^N)}\leq C' \|\psi_0\|_{H^s(\mathbb{R}^N)},\quad \forall t\in[0,1].$$
Let $\widetilde{C}=\max\{C,C'\}$. Since $\|\phi(t)\|_{H^s(\mathbb{R}^N)}<1$ in $[0,\varepsilon^{\delta_n})$, we use \eqref{eq:align} and find that
\begin{equation}\begin{split}\label{limit5}
1& =  \|\phi(\varepsilon^{\delta_n}) \|_{H^s(\mathbb{R}^N)}^2 \\
&<  2\left\|\exp\left(-i \delta_n \varepsilon^{\delta_n}\left(-\Delta+\frac{u}{{\delta_n}} Q_j\right)\right)\psi_0 - \exp(-i {\varepsilon^{\delta_n}} uQ_j)\psi_0 \right\|_{H^s(\mathbb{R}^N)}^2\\
&+2\widetilde{C}\kappa^2  \delta_n^2  \big(\varepsilon^{\delta_n}\big)^2    \left( \|\psi_0\|_{H^s(\mathbb{R}^N)}^{4p+2}  +1\right).
 \end{split}
\end{equation}
If we set $\tilde \delta_n:={\delta_n} \varepsilon^{\delta_n}$ and $u_n=\varepsilon^{\delta_n} u$, then $\tilde\delta_n\xrightarrow{n\rightarrow +\infty} 0$ and $(u_n)_{n\in\N}\in\ell^\infty(\N,\R)$. Then, thanks to \eqref{limit5} and Lemma \ref{coro:main-tool-bis}, there exists $n$ sufficiently large such that
$$2\left\|\exp\left(-i {\tilde \delta_n} \left(-\Delta+\frac{u_n}{{\tilde \delta_n}} Q_j\right)\right)\psi_0-\exp(-i u_n Q_j)\psi_0 \right\|_{H^s(\mathbb{R}^N)}^2 < \frac{1}{2},$$
and $ 2\widetilde{C}\kappa^2{\tilde \delta_n^2 }    (  \|\psi_0\|_{H^s(\mathbb{R}^N)}^{4p+2}  +1) < 1/2$, contradicting \eqref{limit5}. Hence, there exists $\delta_0>0$ small enough such that $\varepsilon^\delta>1$ for every $0<\delta <\delta_0 $. Finally, $1\in [0,\varepsilon^\delta)\subset[0,\delta^{-1}T^\delta)$ and the result is proved thanks to Lemma \ref{coro:main-tool-bis} since we have
\begin{equation*}\begin{split}
\|\phi(1)\|_{H^s(\mathbb{R}^N)}^2&\leq 2\left\|\exp\left(-i \delta \left(-\Delta+\frac{u}{\delta} Q_j\right)\right)\psi_0-\exp(-i uQ_j)\psi_0 \right\|_{H^s(\mathbb{R}^N)}^2\\
&+ 2\widetilde{C}\kappa^2  \delta^2   \left(  \|\psi_0\|_{H^s(\mathbb{R}^N)}^{4p+2}  +1\right)\xrightarrow{\delta\rightarrow 0^+} 0.
\end{split}
\end{equation*}

\smallskip
\noindent
{\bf Case $j=0$}. We define $\psi,\phi,\varphi$ as in the previous case. Now the operator $-\Delta+\frac{u}{\delta} Q_0$ does not commute with $-\Delta$; to proceed as above, we shall find an alternative way to ensure an inequality of the type of \eqref{eq:align}. In this case, we can write $\phi(t)$ via the Duhamel formula as follows:
\begin{equation}\begin{split}\label{duhamel1}\phi(t)
=\left(\exp ( i \delta t \Delta )-\exp(-i t uQ_0)\right)\psi_0  -i  \delta\int_0^{ t} e^{i\delta( t-r) \Delta }\left(\frac{u}{\delta} Q_0+\kappa|\psi(\delta r)|^{2p}\right)\psi(\delta r) d r . 
\end{split}
\end{equation}
Moreover, we introduce an auxiliary function $\tilde\psi(t):=\exp\left(-i  t \left(-\Delta+\frac{u}{\delta} Q_0\right)\right)\psi_0$ and we can write via the Duhamel formula
\begin{equation}\begin{split}\label{duhamel2}\tilde\psi(\delta t) 
= \exp ( i \delta t \Delta ) \psi_0 -i  {u}\int_0^{  t} e^{i\delta( t-r) \Delta } Q_0\tilde\psi(\delta r) dr.
\end{split}
\end{equation}
We use \eqref{duhamel2} in \eqref{duhamel1} and we obtain for $t<\delta^{-1} T^\delta $ that
\begin{equation}\begin{split}\label{duhamel3}\phi(t) =&\left(\tilde \psi(\delta t)-\exp(-i t uQ_0)\psi_0\right) \\
&- i u  \int_0^{  t} e^{i\delta( t-r) \Delta } Q_0\left(\psi(\delta r) -\tilde \psi(\delta r) \right)dr -i\kappa\delta \int_0^{ t} e^{i\delta( t-r) \Delta }|\psi(\delta r)|^{2p}\psi(\delta r) dr \\
=&\left(\tilde \psi(\delta t)-\exp(-i t uQ_0)\psi_0\right)  -i u  \int_0^{  t} e^{i\delta( t-r) \Delta }  Q_0\left(\exp(-i r uQ_0)\psi_0 -\tilde \psi(\delta r) \right)dr \\
&-i u \int_0^{  t} e^{i\delta( t-r) \Delta } Q_0\phi(r)d\tau -i\kappa \delta\int_0^{ t} e^{i\delta( t-r) \Delta }|\psi(\delta r)|^{2p}\psi(\delta r) dr .\\
\end{split}
\end{equation}
We now use \eqref{eq:cauchy} and the unitarity of $e^{it\Delta}$ in $H^s$, and obtain
\begin{equation}\begin{split}\label{limitbis1}\|(-\Delta)^\frac{s}{2}\phi(t)\|_{L^2(\R^N)}^2&\leq 4\left \|(-\Delta)^\frac{s}{2}\left(\tilde \psi(\delta t)-\exp(-i t uQ_0)\psi_0\right)\right \|_{L^2(\R^N)}^2 \\
&+4 u^2  t \int_0^{  t} \left\|(-\Delta)^\frac{s}{2}   Q_0\left(\exp(-i r uQ_0)\psi_0 -\tilde \psi(\delta r)\right) \right\|_{L^2(\R^N)}^2dr \\
&+4 u^2 t  \int_0^{  t}\left\|(-\Delta)^\frac{s}{2}  Q_0\phi(r)  \right\|_{L^2(\R^N)}^2d\tau \\
&+ 4 \kappa^2 \delta^2 t \int_0^{ t} \left\|(-\Delta)^\frac{s}{2} |\psi(\delta r)|^{2p}\psi(\delta r) \right\|_{L^2(\R^N)}^2 dr .\\
\end{split}
\end{equation}
Since $Q_0\in H^s(\R^N)$, there exists $C=C(Q_0,u,\kappa)>0$ such that
\begin{equation}\begin{split}\label{limitbis2}\| \phi(t)\|_{H^s(\R^N)}^2&\leq \left \| \tilde \psi(\delta t)-\exp(-i t uQ_0)\psi_0 \right \|_{H^s(\R^N)}^2 \\
&+C t \int_0^{  t} \left\|   \exp(-i r uQ_0)\psi_0 -\tilde \psi(\delta r)   \right\|_{H^s(\R^N)}^2dr \\
&+C t  \int_0^{  t}\left\|  \phi(r)  \right\|_{H^s(\R^N)}^2d\tau + C \delta^2 t \int_0^{ t} \left\|  \psi(\delta r) \right\|_{H^s(\R^N)}^{4p+2} dr\\
&\leq C\left(1+ t^2\right)\sup_{r \in (0,t)}\left \|\tilde \psi(\delta r)-\exp(-i r uQ_0) \psi_0\right \|_{H^s(\R^N)}^2 \\
&+C t  \int_0^{  t}\left\|  \phi(r)  \right\|_{H^s(\R^N)}^2d\tau + C \delta^2 t \int_0^{ t} \left\|  \phi(r) +  \varphi(r) \right\|_{H^s(\R^N)}^{4p+2} dr.\\
\end{split}
\end{equation}
As in the previous case, we introduce  $\varepsilon^\delta=\sup\{t<\delta^{-1} T^\delta:\ \|\phi(t)\|_{H^s(\R^N)}<1 \}$ and, for $t\leq\min\{1,\varepsilon^\delta\}$, we have (for some new constant $\widetilde{C}=\widetilde{C}(Q_0,u,\kappa)$)
\begin{equation}\begin{split}\label{limitbis3}\| \phi(t)\|_{H^s(\R^N)}^2&\leq 2\widetilde{C}\sup_{r \in(0,1)}\left \| \tilde \psi(\delta r)-\exp(-i r uQ_0) \psi_0\right \|_{H^s(\R^N)}^2 \\
&+\widetilde{C}  (1+\delta^2)   \int_0^{  t}\left\|  \phi(r)  \right\|_{H^s(\R^N)}^2d\tau + \widetilde{C} \delta^2  \int_0^{ 1} \left\|  \varphi(r) \right\|_{H^s(\R^N)}^{4p+2} dr.
\end{split}
\end{equation}
Using the Gronwall inequality we thus obtain
\begin{equation}\begin{split}\label{limitbis4}&\| \phi(t)\|_{H^s(\R^N)}^2 \\ &\leq  e^{\widetilde{C} t (1+\delta^2)} \Bigg(2\widetilde{C} \sup_{r \in (0,1)}\left \| \tilde \psi(\delta r)-\exp(-i r uQ_0) \psi_0\right \|_{H^s(\R^N)}^2   \!\!\!\!\!\!\!\!\!\!\!\!+  \widetilde{C}\delta^2   \int_0^{ 1} \left\|  \varphi(r) \right\|_{H^s(\R^N)}^{4p+2} dr\Bigg) \\
&=  e^{\widetilde{C} t  (1+\delta^2)} \Big(2\widetilde{C} \sup_{r \in (0,1)}\left \|\left(\exp \left(- i \delta r \left(-\Delta+\frac{u}{\delta} Q_0\right) \right)-\exp(-i r uQ_0)\right)\psi_0\right \|_{H^s(\R^N)}^2  \\
&+  \widetilde{C}\delta^2   \int_0^{ 1} \left\|  \varphi(r) \right\|_{H^s(\R^N)}^{4p+2} dr\Big).\\
\end{split}
\end{equation}
Finally, arguing by contradiction exactly as for the case $j>0$, by using \eqref{limitbis4} instead of \eqref{eq:align} and by referring to Corollary \ref{nuovo1} instead of Lemma \ref{coro:main-tool-bis}, we find that there exists $\delta_0>0$ small enough such that $\varepsilon^\delta>1$ for every $0<\delta <\delta_0 $. Then, \eqref{limitbis4} and Corollary \ref{nuovo1} imply that $\|\phi(1)\|_{H^s}\to 0$ as $\delta\to 0$.
\end{proof}

\section{Saturation}\label{sec:saturation} Recall that the 1D Hermite functions are defined for any $n\in\mathbb{N}$ as
\begin{equation}\label{def:hermite}
h_n(x)=(-1)^n(2^nn!\sqrt{\pi})^{-1/2}e^{x^2/2}\frac{d^n}{dx^n}e^{-x^2}.
\end{equation}
In $N$ dimensions, we consider the tensor products of 1-D Hermite functions:
$$h_{n_1,\dots,n_N}(x_1,\dots,x_N)=h_{n_1}(x_1)\dots h_{n_N}(x_N).$$
It is well-known that the Hermite functions
form an (orthonormal) Hilbert basis of $L^2(\mathbb{R}^N,\mathbb{R})$. Moreover, we have the following.
\begin{lemma}\label{lem:density2}For any $s\geq 0$, one has that 
$$
\overline{\rm span}_\mathbb{R}\{h_{n_1,\dots,n_N},n_1,\dots,n_N\in\mathbb{N}\}=H^s(\mathbb{R}^N,\mathbb{R}),$$
 where the closure is taken w.r.t. the $H^s$-norm.
\end{lemma}
\begin{proof}
Let $f\in H^s(\R^N,\R)$ be such that 
$$\langle(1+(-\Delta)^{s/2})f,(1+(-\Delta)^{s/2})h_{\vec{n}} \rangle_{L^2}=0,\quad \forall n\in\N.$$ 
We prove the statement by showing that, necessarily, $f=0$. Denote by $\mathcal{F}$ the Fourier transform: being an isometry in $L^2$, we have that 
$$\langle \mathcal{F}[(1+(-\Delta)^{s/2})f],\mathcal{F}[(1+(-\Delta)^{s/2})h_{\vec{n}}]\rangle_{L^2}=0,$$
We then compute
\begin{align*} &\langle \mathcal{F}[(1+(-\Delta)^{s/2})f],\mathcal{F}[(1+(-\Delta)^{s/2})h_{\vec{n}}]\rangle_{L^2}\\=&\langle(1+|\lambda|^{s/2})\mathcal{F}[f] ,(1+|\lambda|^{s/2})\mathcal{F}[h_{\vec{n}}] \rangle_{L^2} \\
=&(-i)^{\sum_{j=1}^N n_j}\langle(1+|\lambda|^{s/2})\widehat{f} ,(1+|\lambda|^{s/2})h_{\vec{n}} \rangle_{L^2},
\end{align*}
where we used that $\mathcal{F}[h_{\vec{n}}]=(-i)^{\sum_{j=1}^N n_j}h_{\vec{n}}$. We then have that $\widehat{f}:=\mathcal{F}[f]$ is in the orthogonal complement of ${\rm span}_\mathbb{R}\{h_{\vec{n}},\vec{n}\in\mathbb{N}^N\}$ w.r.t. the $L^2$-scalar product associated with the measure $(1+|\lambda|^{s/2})^2d\lambda$. Since the set of linear combinations of Hermite functions is dense in the $L^2$-space (also w.r.t. the weight $(1+|\lambda|^{s/2})^2$), we have that $\mathcal{F}[f]=0$. Hence, we conclude that $f=0$.
\end{proof}

 We introduce now an increasing sequence of vector subspaces of $L^2(\mathbb{R}^N,i\mathbb{R})$: define
$$\mathcal{H}_0={\rm span}_\mathbb{R}\{ih_{0,\dots,0}\}, $$
and then iteratively, for $j\geq 1$, $\mathcal{H}_j$ as the largest vector space whose elements can be written as
$$\phi_0+i\sum_{k=1}^N P_k\phi_k,\quad \phi_0,\phi_k\in\mathcal{H}_{j-1}. $$
Finally, we define the saturation space as $\mathcal{H}_\infty=\cup_{j=0}^\infty \mathcal{H}_j$.
\begin{lemma}\label{lem:density}
The vector space $\mathcal{H}_\infty$ is dense in $H^s(\mathbb{R}^N,i\mathbb{R}), s\geq 0$.
\end{lemma}
\begin{proof}
We prove it for $N=1$ and the extension to higher dimensions is straightforward thanks to the tensor product structure. Thanks to Lemma \ref{lem:density2}, it is enough to show that 
\begin{equation}\label{eq:Hinfty}
\overline{\rm span}_{\mathbb{R}}\{ih_n,n\in\mathbb{N}\}\subset \mathcal{H}_\infty.
\end{equation}
From the recurrence relations
$$Ph_0=-\frac{i}{\sqrt{2}}h_1,\quad Ph_n=\sqrt{\frac{n}{2}}ih_{n-1}-\sqrt{\frac{n+1}{2}}ih_{n+1},\quad n\geq 1, $$
which can be derived straightforwardly by induction (using the definition \eqref{def:hermite}), we have
$$ih_1=iP(\sqrt{2}ih_0)\in\mathcal{H}_1,\quad ih_{n+1}=\sqrt{\frac{n}{n+1}}ih_{n-1}+iP\left(\sqrt{\frac{2}{n+1}}ih_n\right)\in\mathcal{H}_{n+1}, n\geq1, $$
and the conclusion follows.
\end{proof}


\section{Proof of Theorem \ref{thm:main-result}}\label{sec:proof}

Let us also recall that the concatenation $v*u$ of two scalar control laws $u:[0,T_1]\to \mathbb{R}^{N+1},v:[0,T_2]\to \mathbb{R}^{N+1} $ is the scalar control law defined on $[0,T_1+T_2]$ as follows
$$(v*u)(t)=\begin{cases}
u(t), & t\in[0,T_1]\\
v(t-T_1), & t\in(T_1,T_1+T_2],
\end{cases} $$
and the definition extends to controls with values in $\mathbb{R}^{N+1}$ componentwise. Note also that
$$\mathcal{R}(T_1+t,\psi_0,v*u)=\mathcal{R}(t,\mathcal{R}(T_1,\psi_0,u),v),\quad t>0. $$

Consider now the following property:
\begin{itemize}

\smallskip

\item[($P_n$)] Let $\psi_0\in H^s(\mathbb{R}^N,\mathbb{C})$ and $\phi\in\mathcal{H}_n$. For any $\varepsilon,T>0$, there exist $\tau\in[0,T)$ and $(u_0,u):[0,\tau]\to \mathbb{R}^{N+1}$ piecewise constant such that 
the solution $\psi(t;\psi_0)$ of \eqref{eq:schro} associated with the control $(u_0,u)$ and with the initial condition $\psi_0$ satisfies
\begin{align}\label{Pn}\|\psi(\tau;\psi_0)-e^{\phi}\psi_0\|_{H^s(\mathbb{R}^N)}< \varepsilon. \end{align}
\end{itemize}
It is clear that the validity of $(P_n)$ for every $n\in\N$ implies Theorem \ref{thm:main-result}, thanks to the density property proved in Lemma \ref{lem:density}. We are thus left to prove $(P_n)$: we do it by induction.

\textbf{Basis of induction: $n=0$}\\
If $\phi\in\mathcal{H}_0$, there exists $\alpha\in\mathbb{R}$ such that $\phi(x)=i\alpha h_{0,\dots,0}(x)$. Consider then the solution $\mathcal{R}(t,\psi_0,(-\alpha/\delta,0))$ of \eqref{eq:schro} associated with the constant control $(u_0,u)^{\delta,\alpha}:=(-\alpha/\delta,0)\in\mathbb{R}^{N+1}$ and with the initial condition $\psi_0$.
Applying Proposition \ref{lemma:limit-nonlinear1}, we find $\delta\in[0,T)$ such that
$$\|\mathcal{R}(\delta,\psi_0,(-\alpha/\delta,0))-\exp(i\alpha h_{0,\dots,0})\psi_0\|_{H^s(\R^N)}<\varepsilon. $$ 

\textbf{Inductive step: $n\Rightarrow n+1$}\\
Assuming that $(P_n)$ holds, we prove $(P_{n+1})$. If $\phi\in\mathcal{H}_{n+1}$, there exist $(\phi_j)_{j=0,..,N}\subset\mathcal{H}_n$ such that 
$$\phi=\phi_0+i\sum_{j=1}^N P_j\phi_j.$$ Let us start by considering the term $\phi_1$: thanks to Proposition \ref{lemma:main-tool}, we can fix $\gamma\in[0,T/3)$ small enough such that
$$\left\|\exp\left(\frac{\phi_1}{\gamma}\right)\exp(-i\gamma P_1)\exp\left( -\frac{\phi_1}{\gamma}\right)\psi_0-\exp(iP_1\phi_1)\psi_0\right\|_{H^s(\R^N)}< \varepsilon/2.  $$

Thanks to the inductive hypothesis, for any $\epsilon,T,\gamma>0$, there exist $\delta_1\in[0,T/3)$ and a piecewise constant control $(u_0,u)^{\delta_1,\gamma}:[0,\delta_1]\to \mathbb{R}^{N+1}$ such that 
\begin{equation}\label{eq:first-impulsion}
 \left\|
\mathcal{R}(\delta,\psi_0,(u_0,u)^{\delta_1,\gamma})-\exp\left(-\frac{\phi_1}{\gamma}\right)\psi_0\right\|_{H^s(\R^N)}< \epsilon. 
\end{equation}
We now consider a constant control $(u_0,u)^{\delta_2,\gamma}=(0,e_1\gamma/\delta_2):[0,\delta_2]\to\mathbb{R}^{N+1}$: thanks to Proposition \ref{lemma:limit-nonlinear1}, we can find $\delta_2\in[0,T/3)$ such that
\begin{align*}
& \left\|
\mathcal{R}(\delta_1+\delta_2,\psi_0,(u_0,u)^{\delta_2,\gamma}*(u_0,u)^{\delta_1,\gamma})-\exp(-i\gamma P_1)\exp\left(-\frac{\phi_1}{\gamma}\right)\psi_0\right\|_{H^s(\R^N)}\\
\leq &\left\|
\mathcal{R}(\delta_2,\mathcal{R}(\delta_1,\psi_0,(u_0,u)^{\delta_1,\gamma}),(u_0,u)^{\delta_2,\gamma})\!-\!\exp(-i\gamma P_1)\mathcal{R}(\delta_1,\psi_0,(u_0,u)^{\delta_1,\gamma})\right\|_{H^s(\R^N)}\\
+ & \left\|
\exp(-i\gamma P_1)\mathcal{R}(\delta_1,\psi_0,(u_0,u)^{\delta_1,\gamma})-\exp(-i\gamma P_1)\exp\left(-\frac{\phi_1}{\gamma}\right)\psi_0\right\|_{H^s(\R^N)}<2\epsilon.
\end{align*}
Now, we use again the inductive hypothesis to deduce that there exist $\delta_3\in[0,T/3)$ and a piecewise constant control $(u_0,u)^{\delta_3,\gamma}:[0,\delta_3]\to \mathbb{R}^{2}$ such that
\begin{align*}
&\Big\|\mathcal{R}\left(\delta_3,\exp(-i\gamma P_1)\!\exp\!\left(\!\!-\frac{\phi_1}{\gamma}\right)\!\psi_0,(u_0,u)^{\delta_3,\gamma}\right)\!\\
&-\!\exp\!\left(\frac{\phi_1}{\gamma}\right)\!\exp\!(-i\gamma P_1)\!\exp\!\left(-\frac{\phi_1}{\gamma}\right)\!\psi_0\Big\|_{H^s(\R^N)}
<\epsilon. 
\end{align*}
Then, thanks to Proposition \ref{prop:well}, there exists $C>0$ such that
\begin{align*}
& \left\|
\mathcal{R}(\delta_1+\delta_2+\delta_3,\psi_0,(u_0,u)^{\delta_3,\gamma}*(u_0,u)^{\delta_2,\gamma}*(u_0,u)^{\delta_1,\gamma})-\exp(iP_1\phi_1)\psi_0\right\|_{H^s(\R^N)}\\
\leq & \Big\| \mathcal{R}(\delta_3,\mathcal{R}(\delta_1+\delta_2,\psi_0,(u_0,u)^{\delta_2,\gamma}*(u_0,u)^{\delta_1,\gamma}),(u_0,u)^{\delta_3,\gamma})\\
&-\mathcal{R}\left(\delta_3,\exp(-i\gamma P_1)\exp\left(-\frac{\phi_1}{\gamma}\right)\psi_0,(u_0,u)^{\delta_3,\gamma}\right) \Big\|_{H^s(\R^N)}\\
+&\Big\|\mathcal{R}\left(\delta_3,\exp(-i\gamma P_1)\exp\left(-\frac{\phi_1}{\gamma}\right)\psi_0,(u_0,u)^{\delta_3,\gamma}\right)\\
&-\exp\left(\frac{\phi_1}{\gamma}\right)\exp(-i\gamma P_1)\exp\left(-\frac{\phi_1}{\gamma}\right)\psi_0\Big\|_{H^s(\R^N)}\\
+&\left\|\exp\left(\frac{\phi_1}{\gamma}\right)\exp(-i\gamma P_1)\exp\left( -\frac{\phi_1}{\gamma}\right)\psi_0-\exp(iP_1\phi_1)\psi_0\right\|_{H^s(\R^N)}\\
< & C\epsilon+2\epsilon+\varepsilon/2.
\end{align*}

Choosing $\epsilon>0$ small enough such that $C\epsilon+2\epsilon<\varepsilon/2$, we have then proved that the piecewise constant control $$(u_0,u):=(u_0,u)^{\delta_3,\gamma}*(u_0,u)^{\delta_2,\gamma}*(u_0,u)^{\delta_1,\gamma}$$ steers $\psi_0$, the initial state, $\varepsilon$-close to the target $\exp(iP_1\phi_1)\psi_0$ in the time $\tau:=\delta_1+\delta_2+\delta_3<T$.

Finally, the argument for generating the other $\exp(iP_j\phi_j)\psi_0$ is completely identical. To conclude, by the inductive hypothesis, there exists a piecewise constant control $(u_0,u)$ steering the state $\exp(i\sum_{j=1}^NP_j\phi_j)\psi_0$ arbitrarily close to the state 
$$\exp(\phi_0) \exp\left(i\sum_{j=1}^NP_j\phi_j\right)\psi_0=\exp\left(\phi_0+i\sum_{j=1}^NP_j\phi_j\right)\psi_0=\exp(\phi)\psi_0$$
 in arbitrarily small times. This fact concludes the proof of the property $(P_n)$.\\

Corollary \ref{thm:eigenmodes} is then a direct consequence of Theorem \ref{thm:main-result}: it suffices to consider $$\phi(x)=(\nu-\xi)x\rho_S(x),$$
where $\rho_S$ is any smooth function with compact support such that $S\subset{\rm supp}(\rho_S)$ and $\rho_S(x)= 1$ on $S$.

\bibliographystyle{siamplain}
\bibliography{references}

\end{document}